\documentclass[12pt, a4paper]{article}
\usepackage{amsmath}
\usepackage{amsfonts}
\usepackage{dsfont}

\usepackage{mathrsfs}
\usepackage{bbm}
\usepackage{latexsym}
\usepackage{graphicx}
\usepackage{amssymb}
\usepackage{curves}

\newtheorem{theorem}{Theorem}[section]
\newtheorem{lemma}[theorem]{Lemma}
\newtheorem{corollary}[theorem]{Corollary}

\newtheorem{remark}[theorem]{Remark}

\title{{\bf On the second largest distance eigenvalue of a graph}}

\author{ Ruifang Liu$^{a}$\thanks{Supported by NSFC (No.~11201432) and NSF-Henan (Nos.~15A110003 and 15IRTSTHN006). E-mail address:~rfliu@zzu.edu.cn
(R.Liu).} ~ Jie Xue$^{a}$ ~ Litao Guo$^{b}$\thanks{ Supported by
NSFC (No.~11301440) and NSF-Fujian (No.~JA13240).}
\\ ~ \\
{\footnotesize $^a$ School of Mathematics and Statistics, Zhengzhou
University,
Zhengzhou, Henan 450001, China}\\
{\footnotesize $^b$ School of Applied Mathematics, Xiamen University
of Technology, Xiamen, Fujian 361024, China }}
\date{}

\begin{document}
%\openup 1.0\jot
\maketitle

\begin{abstract}
Let $G$ be a simple connected graph of order $n$ and $D(G)$ be the
distance matrix of $G.$ Suppose that
$\lambda_{1}(D(G))\geq\lambda_{2}(D(G))\geq\cdots\geq\lambda_{n}(D(G))$
are the distance spectrum of $G$. A graph $G$ is said to be
determined by its $D$-spectrum if with respect to the distance
matrix $D(G)$, any graph with the same spectrum as $G$ is isomorphic
to $G$. In this paper, we consider spectral characterization on the
second largest distance eigenvalue $\lambda_{2}(D(G))$ of graphs,
and prove that the graphs with
$\lambda_{2}(D(G))\leq\frac{17-\sqrt{329}}{2}\approx-0.5692$ are
determined by their $D$-spectra.


\end{abstract}

\section{Introduction}

~~~~All graphs considered here are simple, undirected and connected.
Let $G$ be a graph with vertex set
$V(G)=\{v_{1},v_{2},\ldots,v_{n}\}$ and edge set $E(G)$. Two
vertices $u$ and $v$ are called adjacent if they are connected by an
edge. Let $d_{G}(v)$ and $N_{G}(v)$ denote the degree and the
neighbor set of a vertex $v$ in $G$, respectively. The distance
between vertices $u$ and $v$ of a graph $G$ is denoted by
$d_G(u,v)$. The diameter of $G,$ denoted by $d$ or $d(G),$ is the
maximum distance between any pair of vertices of $G.$ Let $X$ and
$Y$ be subsets of vertices of $G$. The induced subgraph $G[X]$ is
the subgraph of $G$ whose vertex set is $X$ and whose edge set
consists of all edges of $G$ which have both ends in $X$. For any
$v\in V(G),$ denote by $G-v$ the induced subgraph $G[V\setminus \{v
\}].$ The complete product $G_{1}\bigtriangledown G_{2}$ of graphs
$G_{1}$ and $G_{2}$ is the graph obtained from $G_{1}\cup G_{2}$ by
joining every vertex of $G_{1}$ with every vertex of $G_{2}.$

The distance matrix $D(G)=(d_{ij})_{n\times n}$ of a connected graph
$G$ is the matrix indexed by the vertices of $G,$ where $d_{ij}$
denotes the distance between the vertices $v_{i}$ and $v_{j}$. Let
$\lambda_{1}(D(G))\geq\lambda_{2}(D(G))\geq\cdots\geq\lambda_{n}(D(G))$
be the distance spectrum of $G,$ where $\lambda_{2}(D(G))$ is called
the second largest distance eigenvalue of $G.$ The polynomial
$P_{D}(\lambda)=det(\lambda I-D(G))$ is defined as the distance
characteristic polynomial of the graph $G.$ Two graphs are said to
be $D$-cospectral if they have the same distance spectrum. A graph
$G$ is said to be determined by the $D$-spectra if there is no other
nonisomorphic graph $D$-cospectral to $G$.

Which graphs are determined by their spectrum seems to be a
difficult and interesting problem in the theory of graph spectra.
This question was proposed by Dam and Haemers in \cite{ER1}. In this
paper, Dam and Haemers investigated the cospectrality of graphs up
to order 11. Later, Dam et al. \cite{ER2,ER3} provided two excellent
surveys on this topic. Up to now, only a few families of graphs were
shown to be determined by their spectra. In particular, there are
much fewer results on which graphs are determined by their
$D$-spectra. In \cite{LHQ}, Lin et al. proved that the complete
bipartite graph $K_{n_{1},n_{2}}$ and the complete split graph
$K_{a}\bigtriangledown K_{b}^{c}$ are determined by $D$-spectra, and
conjecture that the complete $k$-partite graph $K_{n_{1}, n_{2},
\ldots, n_{k}}$ is determined by its $D$-spectrum. Recently, Jin and
Zhang \cite{XDZ} have confirmed the conjecture. Xue and Liu
\cite{XL} showed that some special graphs $K^{h}_{n},$ $K^{s+t}_{n}$
and $K^{s,t}_{n}$ are determined by their $D$-spectra. Lin, Zhai and
Gong \cite{LZG} characterized all connected graphs with
$\lambda_{n-1}(D(G))=-1$, and showed that these graphs are
determined by their $D$-spectra. Moreover, in this paper, they also
proved that the graphs with $\lambda_{n-2}(D(G))>-1$ are determined
by their distance spectra. In this paper, we prove that the graphs
with $\lambda_{2}(D(G))\leq\frac{17-\sqrt{329}}{2}\approx-0.5692$
are determined by their $D$-spectra.

Next, we introduce two kinds of graphs $K_{s}^{t}$ and
$K^{n_{1},n_{2},\ldots, n_{k}}_{n}$ as shown in Fig. 1.

$\bullet$ $K_{s}^{t}$: the graph obtained by adding a pendant edge
to $t$ vertices of $K_{s}$, where $2\leq t\leq s$ and $n=s+t$.

$\bullet$ $K^{n_{1},n_{2},\ldots, n_{k}}_{n}$: the graph $G$ with
$d_{G}(v)=n-1$ and $G-v$ is the disjoint union of some complete
graphs, where $n=\sum_{i=1}^{k}n_{i}+1$ and $k\geq2.$

\begin{center}
\setlength{\unitlength}{1.0mm}
\begin{picture}(120,20)

\put(25,0){\bigcircle{20}}
\put(14,18){\circle*{1}}
\curve(14,18,20,8.5)
\put(36,18){\circle*{1}}
\curve(36,18,30,8.5)
\put(14,-18){\circle*{1}}
\curve(14,-18,20,-8.5)
\put(36,-18){\circle*{1}}
\curve(36,-18,30,-8.5)
\put(35,0){\line(1,0){10}}\put(45,0){\circle*{1}}
\put(15,0){\line(-1,0){10}}\put(5,0){\circle*{1}}
\put(23,-1){$K_{s}$}

\put(95,0){\circle*{1}}
\put(80,13){\bigcircle{10}}\put(77,12){$K_{n_{2}}$}
\put(80,-13){\bigcircle{10}}\put(77,-14){$K_{n_{1}}$}
\put(110,13){\bigcircle{10}}\put(107,12){$K_{n_{3}}$}
\put(110,-13){\bigcircle{10}}\put(107,-14){$K_{n_{k}}$}

\curve(84,16,95,0)\curve(106,16,95,0)\curve(84,-16,95,0)\curve(106,-16,95,0)
\curve(78,8.4,95,0)\curve(112,8.4,95,0)\curve(78,-8.4,95,0)\curve(112,-8.4,95,0)
\put(110,0){\circle*{0.55}}\put(110,3){\circle*{0.55}}\put(110,-3){\circle*{0.55}}
\put(94,-5){$v$}

\end{picture}
\vskip 2.5cm  Fig. $1$. Graphs $K_{s}^{t}$ and $K^{n_{1},n_{2},\ldots,
n_{k}}_{n}$.
\end{center}

\section{Preliminaries}

~~~~Before presenting the proof of the main result, we give some
important lemmas and theorems. The following lemma is well-known
Cauchy Interlace Theorem.

\begin{lemma}{\bf (\cite{DMC})} \label{le1}
Let $A$ be a Hermitian matrix of order $n$ with eigenvalues
$\lambda_{1}(A)\geq\lambda_{2}(A)\geq\cdots \geq\lambda_{n}(A),$ and
$B$ be a principal submatrix of $A$ of order $m$ with eigenvalues
$\mu_{1}(B)\geq \mu_{2}(B)\geq\cdots \geq\mu_{m}(B)$. Then
$\lambda_{n-m+i}(A)\leq\mu_{i}(B)\leq\lambda_{i}(A)$ for
$i=1,2,\ldots,m.$
\end{lemma}

Applying Lemma \ref{le1} to the distance matrix $D$ of a graph, we
have

\begin{lemma}\label{le2}
Let $G$ be a graph of order $n$ with distance spectrum
$\lambda_{1}(D(G))\geq\lambda_{2}(D(G))\geq\cdots
\geq\lambda_{n}(D(G)),$ and $H$ be an induced subgraph of $G$ on $m$
vertices with the distance spectrum $\mu_{1}(D(H))\geq
\mu_{2}(D(H))\geq\cdots \geq\mu_{m}(D(H)).$ Moreover, if $D(H)$ is a
principal submatrix of $D(G),$ then
$\lambda_{n-m+i}(D(G))\leq\mu_{i}(D(H))\leq\lambda_{i}(D(G))$ for
$i=1,2,\ldots,m.$
\end{lemma}

\begin{theorem}\label{th20}
Let $G$ be a connected graph and $D(G)$ be the distance matrix of
$G.$ Then $\lambda_{2}(D(G))\geq-1,$ with the equality if and only
if $G\cong K_{n}.$
\end{theorem}

\begin{proof}
Let $G$ be a connected graph with order $n\geq2,$ then $P_{2}$ is an
induced subgraph of $G,$ and $D(P_{2})$ is a principal matrix of
$D(G).$ Note that $\lambda_{2}(D(P_{2}))=-1,$ then by Lemma
\ref{le2}, $\lambda_{2}(D(G))\geq \lambda_{2}(D(P_{2}))=-1.$

For the equality, we only need to prove the necessity. Suppose that
$G$ is not complete graph, then $d(G)\geq 2.$ As a result,
$D(P_{3})$ is a principal submatrix of $D(G).$ Note that
$\lambda_{2}(D(P_{3}))=1-\sqrt{3},$ by Lemma \ref{le2},
$\lambda_{2}(D(G))\geq
\lambda_{2}(D(P_{3}))=1-\sqrt{3}\approx-0.7321,$ a contradiction.
Hence $G\cong K_{n}.$ \ \ $\Box$
\end{proof}

\begin{remark}\label{re1}
From the proof of the above theorem, there is no graph of order $n$
with $-1<\lambda_{2}(D(G))<1-\sqrt{3}.$
\end{remark}

First we will investigate which graphs satisfy
$\lambda_{2}(D(G))\leq\frac{17-\sqrt{329}}{2}\approx -0.5692$. Let
$G$ be a graph with $\lambda_{2}(D(G))\leq\frac{17-\sqrt{329}}{2}$.
We call $H$ a forbidden subgraph of $G$ if $G$ contains no $H$ as an
induced subgraph.

For any $S\subseteq V(G)$, let $D_{G}(S)$ denote the principal
submatrix of $D(G)$ obtained by $S$.

\begin{lemma}\label{le4}%------
Let $G$ be a connected graph and $D(G)$ be the distance matrix of
$G.$ If $\lambda_{2}(D(G))\leq\frac{17-\sqrt{329}}{2}$, then
$C_{4}$, $C_{5}$, $P_{5}$ and $H_{i}~(i\in\{1,2,3\})$ are forbidden
subgraphs of $G$.
\end{lemma}

\begin{center}
\setlength{\unitlength}{1.0mm}
\begin{picture}(100,45)
\put(5,0){\line(1,0){40}}
\put(5,0){\circle*{1}}\put(15,0){\circle*{1}}\put(25,0){\circle*{1}}\put(35,0){\circle*{1}}\put(45,0){\circle*{1}}
\put(4,2){$v_{1}$}\put(14,2){$v_{2}$}\put(24,2){$v_{3}$}\put(34,2){$v_{4}$}\put(44,2){$v_{5}$}
\put(23,-7){$P_{5}$}

\put(65,0){\line(1,0){30}}
\put(65,0){\circle*{1}}\put(75,0){\circle*{1}}\put(85,0){\circle*{1}}\put(95,0){\circle*{1}}
\curve(70,10,65,0)\curve(70,10,75,0)\put(70,10){\circle*{1}}
\put(62,2){$v_{1}$}\put(75,2){$v_{2}$}\put(84,2){$v_{3}$}\put(94,2){$v_{4}$}\put(71,11){$v_{5}$}
\put(78,-7){$H_{3}$}

\put(0,25){\line(1,0){12}}\put(0,25){\line(0,1){12}}\put(0,37){\line(1,0){12}}\put(12,37){\line(0,-1){12}}
\put(0,25){\circle*{1}}\put(0,37){\circle*{1}}\put(12,25){\circle*{1}}\put(12,37){\circle*{1}}
\put(4,18){$C_{4}$}

\put(24,25){\line(1,0){10}}\curve(22,33,24,25)\curve(22,33,29,38)\curve(36,33,34,25)\curve(36,33,29,38)
\put(24,25){\circle*{1}}\put(22,33){\circle*{1}}\put(29,38){\circle*{1}}\put(36,33){\circle*{1}}\put(34,25){\circle*{1}}
\put(27,18){$C_{5}$}

\put(46,25){\line(1,0){12}}\put(46,25){\line(0,1){12}}\put(46,37){\line(1,0){12}}\put(58,37){\line(0,-1){12}}\put(46,25){\line(1,1){12}}
\put(46,25){\circle*{1}}\put(46,37){\circle*{1}}\put(58,37){\circle*{1}}\put(58,25){\circle*{1}}
\put(50,18){$H_{1}$}

\put(69,25){\line(1,0){30}}\curve(73,35,79,25)
\put(69,25){\circle*{1}}\put(79,25){\circle*{1}}\put(89,25){\circle*{1}}\put(99,25){\circle*{1}}\put(73,35){\circle*{1}}
\put(68,27){$v_{1}$}\put(78,27){$v_{2}$}\put(88,27){$v_{3}$}\put(98,27){$v_{4}$}\put(73,36){$v_{5}$}
\put(82,18){$H_{2}$}

\end{picture}
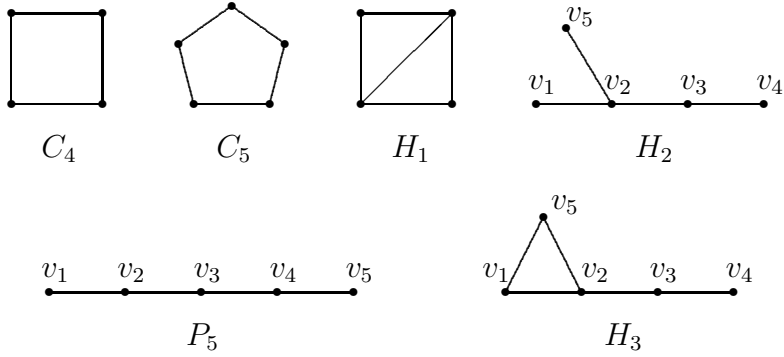
\vskip 1cm  Fig. $2$. Graphs $C_{4}$, $C_{5}$, $P_{5}$ and $H_{1}-H_{3}.$
\end{center}

\begin{proof}
Note that the diameters of graphs $C_{4}$, $C_{5}$ and $H_{1}$ are
all 2. Their $D$-spectra are shown as follows.\\

\begin{tabular}{c|c|c|c|c|c}
\hline
 & $\lambda_{1}$ & $\lambda_{2}$ & $\lambda_{3}$ & $\lambda_{4}$ & $\lambda_{5}$ \\
\hline
$C_{4}$&4.0000&0.0000&-2.0000&-2.0000&\\
\hline
$C_{5}$&6.0000&-0.3820&-0.3820&-2.6180&-2.6180\\
\hline
$H_{1}$&3.5616&-0.5616&-1.0000&-2.0000&\\
\hline
\end{tabular}.\\

If $C_{4}$ is an induced subgraph of $G$, then $D(C_{4})$ is a
principal submatrix of $D(G)$. By Lemma \ref{le2}, then
$\lambda_{2}(D(G))\geq
\lambda_{2}(D(C_{4})>\frac{17-\sqrt{329}}{2}\approx-0.5692$, a
contradiction. Thus $C_{4}$ is a forbidden subgraph of $G$.
Similarly, we can prove that $C_{5}$ and $H_{1}$ are also forbidden
subgraphs of $G$.

Consider $P_{5}$. Suppose that $P_{5}$ is an induced subgraph of
$G$, then $d_{G}(v_{1},v_{5})\in \{2,3,4\}$. If
$d_{G}(v_{1},v_{5})=4$, then
$D_{G}(\{v_{1},v_{2},v_{3},v_{4},v_{5}\})=D(P_{5})$, thus $D(P_{5})$
is a principal submatrix of $D(G)$. By Lemma \ref{le2}, we have
$\lambda_{2}(D(G))\geq
\lambda_{2}(D(P_{5}))=-0.5578>\frac{17-\sqrt{329}}{2}$, a
contradiction. If $d_{G}(v_{1},v_{5})\in \{2,3\}$, let
$d_{G}(v_{1},v_{4})=a$, $d_{G}(v_{1},v_{5})=b$ and
$d_{G}(v_{2},v_{5})=c$, then $a,b,c\in \{2,3\}$. Hence the principal
submatrix of $D(G)$
\begin{equation*}\begin{split}
D_{G}(\{v_{1},v_{2},v_{3},v_{4},v_{5}\})&=\left(\begin{array}{cccccccc}
0&1&2&a&b\\
1&0&1&2&c\\
2&1&0&1&2\\
a&2&1&0&1\\
b&c&2&1&0\\
\end{array}\right).\\
\end{split}\end{equation*}
By a simple calculation, we have\\\\
\begin{tabular}{c|c|c|c|c|c|c|c|c}
\hline
$(a,b,c)$ & $(3,3,3)$ & $(3,2,2)$ & $(3,2,3)$ & $(3,3,2)$ & $(2,3,3)$ & $(2,3,2)$ & $(2,2,2)$ & $(2,2,3)$\\
\hline
$\lambda_{2}$&-0.4348&-0.3260&0&-0.3713&-0.3713&-0.1646&-0.2909&-0.3260\\
\hline
\end{tabular}.\\\\
From Lemma \ref{le2}, this contradicts
$\lambda_{2}(D(G))\leq\frac{17-\sqrt{329}}{2}$. Hence $P_{5}$ is a
forbidden subgraph of $G$.

Consider $H_{2}$. Suppose that $H_{2}$ is an induced subgraph of
$G$. Let $d_{G}(v_{1},v_{4})=a$ and $d_{G}(v_{4},v_{5})=b$, then
$a,b\in \{2,3\}$. We get the principal submatrix of $D(G)$:
\begin{equation*}\begin{split}
D_{G}(\{v_{1},v_{2},v_{3},v_{4},v_{5}\})&=\left(\begin{array}{cccccccc}
0&1&2&a&2\\
1&0&1&2&1\\
2&1&0&1&2\\
a&2&1&0&b\\
2&1&2&b&0\\
\end{array}\right).\\
\end{split}\end{equation*}
By a simple calculation, we have\\\\
\begin{tabular}{c|c|c|c|c}
\hline
$(a,b)$ & $(3,3)$ & $(2,3)$ & $(3,2)$ & $(2,2)$\\
\hline
$\lambda_{2}$&-0.5120&-0.3583&-0.3583&-0.2245\\
\hline
\end{tabular}.\\\\
By Lemma \ref{le2}, $\lambda_{2}(D(G))\geq
\lambda_{2}(D_{G}(\{v_{1},v_{2},v_{3},v_{4},v_{5}\}))>\frac{17-\sqrt{329}}{2}$,
a contradiction. Hence $H_{2}$ is a forbidden subgraph of $G$.

Consider $H_{3}$. Suppose that $H_{3}$ is an induced subgraph of
$G$. Let $d_{G}(v_{1},v_{4})=a$ and $d_{G}(v_{4},v_{5})=b$, then
$a,b\in \{2,3\}$. Then the principal submatrix of $D(G)$
\begin{equation*}\begin{split}
D_{G}(\{v_{1},v_{2},v_{3},v_{4},v_{5}\})&=\left(\begin{array}{cccccccc}
0&1&2&a&1\\
1&0&1&2&1\\
2&1&0&1&2\\
a&2&1&0&b\\
1&1&2&b&0\\
\end{array}\right).\\
\end{split}\end{equation*}
By a simple calculation, we have\\\\
\begin{tabular}{c|c|c|c|c}
\hline
$(a,b)$ & $(3,3)$ & $(2,3)$ & $(3,2)$ & $(2,2)$\\
\hline
$\lambda_{2}$&-0.5686&-0.3626&-0.3626&-0.3311\\
\hline
\end{tabular}.\\\\
From Lemma \ref{le2}, we also get a contradiction since
$\lambda_{2}(D(G))\leq\frac{17-\sqrt{329}}{2}$. Therefore $H_{3}$ is
a forbidden subgraph of $G$. \ \ $\Box$
\end{proof}

Let $\mathbb{K}_{s}^{t}=\{K_{s}^{t}|2\leq t\leq s, s+t=n\}$ and
$\mathbb{K}^{n_{1},n_{2},\ldots, n_{k}}_{n}=\{K^{n_{1},n_{2},\ldots,
n_{k}}_{n}|\sum_{i=1}^{k}n_{i}+1=n, k\geq 2\}.$

\begin{lemma}\label{le5}%------
Let $G$ be a connected graph and $D(G)$ be the distance matrix of
$G.$ If $\lambda_{2}(D(G))\leq\frac{17-\sqrt{329}}{2}$, then $G\in
K_{n}\cup \mathbb{K}_{s}^{t}\cup \mathbb{K}^{n_{1},n_{2},\ldots,
n_{k}}_{n}$.
\end{lemma}

\begin{proof}
Let $\lambda_{2}(D(G))\leq\frac{17-\sqrt{329}}{2}$. Note that
$P_{5}$ is a forbidden subgraph of $G$, then $d(G)\leq 3$. $d(G)=1$
indicates that $G\cong K_{n}.$ Next we consider $d(G)\geq 2$, then
there exist non-adjacent vertices in $G$. For any two non-adjacent
vertices, we claim that they have at most one common neighbor.
Otherwise there exists forbidden subgraph $C_{4}$ or $H_{1}$. Next
we distinguish the following two cases:

{\bf Case $1$.}\ \ There exist two non-adjacent vertices $u$ and $v$
such that $|N_{G}(u)\cap N_{G}(v)|=0$.

Then $d_{G}(u,v)=3$. Let $P=uxyv$ be the shortest path between $u$
and $v$. Thus we have $d_{G}(u)=d_{G}(v)=1.$ Otherwise there exists
forbidden subgraph $P_{5}, H_{3},C_{4}$ or $H_{1}$.

Let $S=(N_{G}(x)\cup N_{G}(y))\backslash\{u,v\}$ and
$T=V(G)\backslash(S\cup \{u,v,x,y\})$.

We claim that $G[S\cup \{x,y\}]$ is a clique, otherwise there exists
forbidden subgraph $H_{2}$ or $H_{1}$. If $T=\emptyset$, then
$G\cong K_{s}^{2}$. If $T\neq \emptyset$, then we claim that each
vertex in $T$ is adjacent to some vertices in $S$. Otherwise the
distance from the vertex to $u$ or $v$ is at least 4, a
contradiction. Furthermore, we claim that each vertex in $T$ is
adjacent to exactly one vertex in $S$ and the vertices in $T$ have
no common neighbor in $S$. Otherwise there exists forbidden subgraph
$H_{1}$, $H_{2}$ or $H_{3}$. Moreover, by forbidden subgraph
$C_{4}$, $T$ is an independent set. Therefore $G\cong K_{s}^{t}$.

{\bf Case $2$.}\ \ For any two non-adjacent vertices, they have
exactly one common neighbor.

Let $S=\{v_{1},v_{2},\ldots,v_{k}\}$ be the maximum independent set
of $G$, where $k\geq 2$. Let $\bar{S}=V(G)\backslash S$. Then by the
definition of the maximum independent set, each vertex in $\bar{S}$
is adjacent to some vertices in $S$.

Let $v$ be the only common neighbor of $v_{1}$ and $v_{2}.$ We claim
that $v$ is adjacent to each vertex in $S.$ Otherwise there exists
forbidden subgraph $H_{2}, C_{5}, C_{4}$ or $H_{1}.$ Since any two
non-adjacent vertices of $G$ have exactly one common neighbor, then
each vertex in $\bar{S}\backslash\{v\}$ is adjacent to exactly one
vertex in $S$.

Let $V_{1}, V_{2}, \ldots, V_{k}$ be a partition of
$\bar{S}\backslash\{v\}$, where $V_{i}$ (possibly empty) is the
vertex subset whose vertices are all adjacent to $v_{i}.$ Next we
show that each vertex in $\bar{S}\backslash\{v\}$ is adjacent to
$v$. Without loss of generality, we consider the vertex subset
$V_{1}$. Let $v_{1}^{\star}\in V_{1}$. Suppose that $v_{1}^{\star}v
\notin E(G).$ Note that $v_{1}^{\star}$ is not adjacent to $v_{2},$
hence we can suppose that $w$ is the only common neighbor of
$v_{1}^{\star}$ and $v_{2}$. Then
$G[vv_{1}v_{1}^{\star}wv_{2}]=C_{5}$ or
$G[vv_{1}v_{1}^{\star}w]=C_{4}$, a contradiction.

Moreover, we claim that $G[V_{i}]~(i=1,2,\ldots,k)$ is a clique and
$E[V_{i}, V_{j}]=\emptyset~(i\neq j).$ Otherwise there exists
forbidden subgraph $H_{1}$. Thus we have $G\cong
K^{n_{1},n_{2},\ldots, n_{k}}_{n}$.\ \ $\Box$
\end{proof}

\section{Main results}

~~~~In this section, we will show that the graphs with
$\lambda_{2}(D(G))\leq\frac{17-\sqrt{329}}{2}$ are determined by
their $D$-spectra. First, we give the distance characteristic
polynomials of $K_{s}^{t}$ and $K^{n_{1},n_{2},\ldots, n_{k}}_{n}$.

\begin{lemma}\label{le11}%------
Let $G=K_{s}^{t},$ where $2\leq t\leq s$ and $n=s+t$. Then the
distance characteristic polynomial of $G$ is as follows.\\
For $s\geq t+1,$ $$P_{D}(\lambda)=(\lambda+1)^{s-t-1}(\lambda+2-\sqrt{2})^{t-1}(\lambda+2+\sqrt{2})^{t-1}[\lambda^{3}+(5-s-3t)\lambda^{2}+(6-4s-2t-st)\lambda+2-2s-st].$$
For $s=t,$
$$P_{D}(\lambda)=(\lambda+2-\sqrt{2})^{t-1}(\lambda+2+\sqrt{2})^{t-1}[\lambda^{2}+(4-4t)\lambda+2-2t-t^{2}].$$
\end{lemma}

\begin{proof} For $s\geq t+1.$ The distance matrix of $K_{s}^{t}$ can be written as
\begin{equation*}\begin{split}
D&=\left(\begin{array}{ccccccc}
\sigma&\alpha&\cdots&\alpha&\beta\\
\alpha&\sigma&\cdots&\alpha&\beta\\
\vdots&\vdots&\ddots&\vdots&\vdots\\
\alpha&\alpha&\cdots&\sigma&\beta\\
\beta^{\top}&\beta^{\top}&\cdots&\beta^{\top}&B\\
\end{array}\right),\\
\end{split}\end{equation*}
where $\sigma=\left(\begin{array}{ccccccc}
0&1\\
1&0\\
\end{array}\right)$,
$\alpha=\left(\begin{array}{ccccccc}
3&2\\
2&1\\
\end{array}\right)$,
$\beta=\left(\begin{array}{ccccccc}
2&2&\cdots &2\\
1&1&\cdots &1\\
\end{array}\right)_{2\times (s-t)},$ and

$B=\left(\begin{array}{ccccccc}
0&1&\cdots &1\\
1&0&\cdots &1\\
\vdots &\vdots &\ddots &\vdots\\
1&1&\cdots &0\\
\end{array}\right)_{(s-t)\times (s-t)}.$

Thus
\begin{equation*}\medmuskip=0mu \begin{split}
&det(\lambda I-D)=\left|\begin{array}{ccccccc}
\lambda I_{2\times 2}-\sigma&-\alpha&\cdots&-\alpha&-\beta\\
-\alpha&\lambda I_{2\times 2}-\sigma &\cdots &-\alpha&-\beta\\
\vdots&\vdots&\ddots&\vdots&\vdots\\
-\alpha &-\alpha &\cdots &\lambda I_{2\times 2}-\sigma &-\beta\\
-\beta^{\top}&-\beta^{\top}&\cdots&-\beta^{\top}&\lambda I_{(s-t)\times (s-t)}-B\\
\end{array}\right|\\
&=\left|\begin{array}{cccccccccc}
\lambda I_{2\times 2}-\sigma-(t-1)\alpha&-\alpha&\cdots&-\alpha&-\beta\\
0&\lambda I_{2\times 2}-\sigma+\alpha &\cdots &0&0\\
\vdots&\vdots&\ddots&\vdots&\vdots\\
0&0&\cdots &\lambda I_{2\times 2}-\sigma+\alpha &0\\
-t\beta^{\top}&-\beta^{\top}&\cdots&-\beta^{\top}&\lambda I_{(s-t)\times (s-t)}-B\\
\end{array}\right|\\
&=|\lambda I_{2\times 2}-\sigma+\alpha |^{t-1}\left|\begin{array}{cccccccccc}
\lambda I_{2\times 2}-\sigma-(t-1)\alpha&-\beta\\
-t\beta^{\top}&\lambda I_{(s-t)\times (s-t)}-B\\
\end{array}\right|\\
&=(\lambda+1)^{s-t-1}(\lambda+2-\sqrt{2})^{t-1}(\lambda+2+\sqrt{2})^{t-1}[\lambda^{3}+(5-s-3t)\lambda^{2}+(6-4s-2t-st)\lambda+2-2s-st].
\end{split}\end{equation*}

For $s=t.$ The distance matrix of $K_{s}^{t}$ can be written as
\begin{equation*}\begin{split}
D&=\left(\begin{array}{ccccccc}
\sigma&\alpha&\cdots&\alpha\\
\alpha&\sigma&\cdots&\alpha\\
\vdots&\vdots&\ddots&\vdots\\
\alpha&\alpha&\cdots&\sigma\\
\end{array}\right).\\
\end{split}\end{equation*}

Thus
\begin{equation*}\medmuskip=0mu \begin{split}
&det(\lambda I-D)=\left|\begin{array}{ccccccc}
\lambda I_{2\times 2}-\sigma&-\alpha&\cdots&-\alpha\\
-\alpha&\lambda I_{2\times 2}-\sigma &\cdots &-\alpha\\
\vdots&\vdots&\ddots&\vdots\\
-\alpha &-\alpha &\cdots &\lambda I_{2\times 2}-\sigma \\
\end{array}\right|\\
&=\left|\begin{array}{cccccccccc}
\lambda I_{2\times 2}-\sigma-(t-1)\alpha&-\alpha&\cdots&-\alpha\\
0&\lambda I_{2\times 2}-\sigma+\alpha &\cdots&0\\
\vdots&\vdots&\ddots&\vdots\\
0&0&\cdots &\lambda I_{2\times 2}-\sigma+\alpha \\
\end{array}\right|\\
&=|\lambda I_{2\times 2}-\sigma+\alpha |^{t-1}
|\lambda I_{2\times 2}-\sigma-(t-1)\alpha|\\
&=(\lambda+2-\sqrt{2})^{t-1}(\lambda+2+\sqrt{2})^{t-1}[\lambda^{2}+(4-4t)\lambda+2-2t-t^{2}].
\end{split}\end{equation*}\ \ $\Box$
\end{proof}

\begin{corollary}\label{co22}%------
Let $G=K_{s}^{t},$ where $3\leq t+1\leq s$ and $n=s+t$. Then $\lambda_{2}(D(G))=\sqrt{2}-2.$
\end{corollary}

\begin{proof}
By Lemma \ref{le11},
$$P_{D}(\lambda)=(\lambda+1)^{s-t-1}(\lambda+2-\sqrt{2})^{t-1}(\lambda+2+\sqrt{2})^{t-1}[\lambda^{3}+(5-s-3t)\lambda^{2}+(6-4s-2t-st)\lambda+2-2s-st].$$
Let $f(\lambda)=\lambda^{3}+(5-s-3t)\lambda^{2}+(6-4s-2t-st)\lambda+2-2s-st.$
By a
simple calculation,
\begin{equation*}\begin{split}
\begin{cases}
f(0)=2-2s-st<0,\\
f(-\frac{2}{3})=-\frac{2}{27}+\frac{2}{9}s-\frac{1}{3}st<0,\\
f(-1)=s-t>0.
\end{cases}
\end{split}\end{equation*}
Note that $f(\lambda)\rightarrow +\infty ~(\lambda\rightarrow +\infty)$ and
$f(0)<0$, so there is at least one root in $(0,+\infty)$. Since
$f(-\frac{2}{3})<0$ and $f(-1)>0$, then there is at least one root
in $(-1,-\frac{2}{3})$. By $f(\lambda)\rightarrow -\infty ~(\lambda\rightarrow
-\infty)$ and $f(-1)>0$, so there is at least one root in
$(-\infty,-1)$. Thus there is exactly one root in each interval. So $\lambda_{2}(D(G))=\sqrt{2}-2.$ \ \ $\Box$
\end{proof}

\begin{theorem}\label{th3}%------
No two non-isomorphic graphs in $\mathbb{K}_{s}^{t}$ are
$D$-cospectral.
\end{theorem}

\begin{proof}
For any $K_{s_{1}}^{t_{1}}, K_{s_{2}}^{t_{2}}\in
\mathbb{K}_{s}^{t}.$ Suppose that they are $D$-cospectral, by
Lemma \ref{le11}, then $t_{1}=t_{2}.$ Note that
$s_{1}+t_{1}=s_{2}+t_{2}=n$, thus $s_{1}=s_{2},$ that is,
$K_{s_{1}}^{t_{1}}\cong K_{s_{2}}^{t_{2}}$. \ \ $\Box$
\end{proof}

\begin{lemma}\label{le12}%------
Let $G=K^{n_{1},n_{2},\ldots, n_{k}}_{n}$, where
$n=\sum_{i=1}^{k}n_{i}+1$ and $k\geq2.$ Then the distance
characteristic polynomial of $G$ is
$$P_{D}(\lambda)=(\lambda+1)^{n-k-1}(\lambda-\sum_{i=1}^{k}\frac{n_{i}(2\lambda+1)}{\lambda+n_{i}+1})\prod_{i=1}^{k}(\lambda+n_{i}+1).$$
\end{lemma}

\begin{proof}
The distance matrix of $G$ has the form
\begin{equation*}\begin{split}
D&=\left(\begin{array}{ccccccc}
J_{n_{1}}-I_{n_{1}}&2J_{n_{1}\times n_{2}}&\cdots&2J_{n_{1}\times n_{k}}&J_{n_{1}\times 1}\\
2J_{n_{2}\times n_{1}}&J_{n_{2}}-I_{n_{2}}&\cdots&2J_{n_{2}\times n_{k}}&J_{n_{2}\times 1}\\
\vdots&\vdots&\ddots&\vdots&\vdots\\
2J_{n_{k}\times n_{1}}&2J_{n_{k}\times n_{2}}&\cdots&J_{n_{k}}-I_{n_{k}}&J_{n_{k}\times 1}\\
J_{1\times n_{1}}&J_{1\times n_{2}}&\cdots&J_{1\times n_{k}}&0\\
\end{array}\right),\\
\end{split}\end{equation*}
where $J_{n_{i}}$ is the all-one square matrix with order $n_{i}$,
$I_{n_{i}}$ is the identity square matrix with order $n_{i}$, and
$J_{n_{i}\times n_{j}}~(i\neq j)$ is the all-one matrix with $n_{i}$
rows and $n_{j}$ columns.

Then
\begin{equation*}\medmuskip=0mu \begin{split}
&det(\lambda I-D)=(\lambda+1)^{n_{1}-1}(\lambda+1)^{n_{2}-1}\cdots (\lambda+1)^{n_{k}-1}\left|\begin{array}{ccccccc}
\lambda -(n_{1}-1)&-2n_{2}&\cdots&-2n_{k}&-1\\
-2n_{1}&\lambda -(n_{2}-1) &\cdots &-2n_{k}&-1\\
\vdots&\vdots&\ddots&\vdots&\vdots\\
-2n_{1}&-2n_{2} &\cdots &\lambda -(n_{k}-1)&-1\\
-n_{1}&-n_{2} &\cdots &-n_{k}&\lambda\\
\end{array}\right|\\
&=(\lambda+1)^{\sum_{i=1}^{k} n_{i}-k}\left|\begin{array}{cccccccccc}
\lambda +n_{1}+1&0&\cdots&0&-1-2\lambda\\
0&\lambda +n_{2}+1&\cdots&0&-1-2\lambda\\
\vdots&\vdots&\ddots&\vdots&\vdots\\
0&0&\cdots&\lambda +n_{k}+1&-1-2\lambda\\
-n_{1}&-n_{2} &\cdots &-n_{k}&\lambda\\
\end{array}\right|\\
&=(\lambda+1)^{n-k-1}(\lambda-\sum_{i=1}^{k}\frac{n_{i}(2\lambda+1)}{\lambda+n_{i}+1})\prod_{i=1}^{k}(\lambda+n_{i}+1).
\end{split}
\end{equation*}\ \ \ $\Box$
 \end{proof}

\begin{theorem}\label{th10}%------
No two non-isomorphic graphs in $\mathbb{K}^{n_{1},n_{2},\ldots,
n_{k}}_{n}$ are $D$-cospectral.
\end{theorem}
\begin{proof}
Let $G=K^{n_{1},n_{2},\ldots, n_{k}}_{n}$. By Lemma \ref{le12},
$-1$ is the eigenvalue of $G$ with multiplicity $n-k-1.$ Suppose
that $G^{\star}\in\mathbb{K}^{n_{1},n_{2},\ldots,n_{k}}_{n}$ is
$D$-cospectral to $G,$ then $-1$ is also the eigenvalue of
$G^{\star}$ with multiplicity $n-k-1$. Hence we may assume that
$G^{\star}=K^{n_{1}^{\star},n_{2}^{\star},\ldots,
n_{k}^{\star}}_{n}$ ($k$ parts) with
$\sum_{i=1}^{k}n_{i}^{\star}=n-1$. Let us denote the rest of
common distance eigenvalues of $G$ and $G^{\star}$ by $\lambda_{1},
\lambda_{2},\ldots, \lambda_{k+1}.$ Let
\begin{equation*}\begin{split}
g(\lambda,n_{1},n_{2},\ldots,n_{k})=&(\lambda-\sum_{i=1}^{k}\frac{n_{i}(2\lambda+1)}{\lambda+n_{i}+1})\prod_{i=1}^{k}(\lambda+n_{i}+1)\\
=&\lambda\prod_{i=1}^{k}(\lambda+n_{i}+1)-(2\lambda+1)\sum_{i=1}^{k}n_{i}\prod_{j=1,j\neq
i}^{k}(\lambda+n_{j}+1)\\
=&\sum_{i=0}^{k+1}\mu_{i}\lambda^{k+1-i},
\end{split}\end{equation*} then $\lambda_{1},
\lambda_{2},\ldots, \lambda_{k+1}$ are the roots of
$g(\lambda,n_{1},n_{2},\ldots,n_{k}).$

Denote
$s_{1}=\sum_{j=1}^{k}n_{j},
s_{2}=\sum_{1\leq j_{1}<j_{2} \leq k}n_{j_{1}}n_{j_{2}},\ldots,
s_{i}=\sum_{1\leq j_{1}< j_{2}<\cdots< j_{i}\leq k}n_{j_{1}}n_{j_{2}}\cdots n_{j_{i}},\ldots,\\ s_{k}=n_{1}n_{2}\cdots n_{k}.$
Obviously, $\mu_{0}=1$ and
$\mu_{1}=-\sum_{i=1}^{k}n_{i}+k=-s_{1}+k.$
%$\mu_{2}=-3s_{2}-ks_{1}+\binom{k}{2}.$
For $2\leq i\leq k$, by the calculation, we have
\begin{equation*}\medmuskip=0mu\begin{split}
\mu_{i}=&\sum_{1\leq j_{1}<j_{2}<\cdots<j_{i}\leq k}(n_{j_{1}}+1)\cdots(n_{j_{i}}+1)-2\sum_{l=1}^{k}n_{l}\sum_{1\leq j_{1}<j_{2}<\cdots<j_{i-1}\leq k,l\not\in\{j_{1},\ldots,j_{i-1}\}}(n_{j_{1}}+1)\cdots(n_{j_{i-1}}+1)\\
&-\sum_{l=1}^{k}n_{l}\sum_{1\leq j_{1}<j_{2}<\cdots<j_{i-2}\leq k,l\not\in\{j_{1},\ldots,j_{i-2}\}}(n_{j_{1}}+1)\cdots(n_{j_{i-2}}+1)\\
=&t_{i}s_{i}+t_{i-1}s_{i-1}+t_{i-2}s_{i-2}+\cdots+t_{1}s_{1}+t_{0},
\end{split}
\end{equation*}
where $t_{i}=1-2i,$ and
$\mu_{k+1}=-ks_{k}-(k-1)s_{k-1}-\cdots-2s_{2}-s_{1}.$

Define $s_{1}^{\star}=\sum_{j=1}^{k}n_{j}^{\star},
s_{2}^{\star}=\sum_{1\leq j_{1}<j_{2} \leq k}n_{j_{1}}^{\star}n_{j_{2}}^{\star},\ldots,
s_{i}^{\star}=\sum_{1\leq j_{1}< j_{2}<\cdots< j_{i}\leq k}n_{j_{1}}^{\star}n_{j_{2}}^{\star}\cdots n_{j_{i}}^{\star},\ldots,\\ s_{k}^{\star}=n_{1}^{\star}n_{2}^{\star}\cdots n_{k}^{\star}.$ Let
\begin{equation*}\begin{split}
g(\lambda,n_{1}^{\star},n_{2}^{\star},\ldots,n_{k}^{\star})=&(\lambda-\sum_{i=1}^{k}\frac{n_{i}^{\star}(2\lambda+1)}{\lambda+n_{i}^{\star}+1})\prod_{i=1}^{k}(\lambda+n_{i}^{\star}+1)\\
=&\lambda\prod_{i=1}^{k}(\lambda+n_{i}^{\star}+1)-(2\lambda+1)\sum_{i=1}^{k}n_{i}^{\star}\prod_{j=1,j\neq
i}^{k}(\lambda+n_{j}^{\star}+1)\\
=&\sum_{i=0}^{k+1}\mu_{i}^{\star}\lambda^{k+1-i}.
\end{split}\end{equation*} Similar to the calculation of $\mu_{i}$,
then
$$\mu_{i}^{\star}=t_{i}s_{i}^{\star}+t_{i-1}s_{i-1}^{\star}+t_{i-2}s_{i-2}^{\star}+\cdots+t_{1}s_{1}^{\star}+t_{0},~~~~~ i=1,2,\ldots,k.$$

Note that $\lambda_{1}, \lambda_{2},\ldots, \lambda_{k+1}$ are also
the roots of
$g(\lambda,n_{1}^{\star},n_{2}^{\star},\ldots,n_{k}^{\star}),$ hence
$\mu_{i}=\mu_{i}^{\star}.$  Recursively, we can obtain that
$s_{1}=s_{1}^{\star}, s_{2}=s_{2}^{\star},\ldots,
s_{k}=s_{k}^{\star}.$ Since $n_{1},n_{2},\ldots,n_{k}$ are the roots
of $x^{k}-s_{1}x^{k-1}+s_{2}x^{k-2}+\cdots+(-1)^{k}s_{k}=0,$ then
$n_{1}^{\star},n_{2}^{\star},\ldots,n_{k}^{\star}$ are also the
roots of the same polynomial, and hence
$n_{1}^{\star},n_{2}^{\star},\ldots,n_{k}^{\star}$ must be a
permutation of $n_{1},n_{2},\ldots,n_{k}.$ So $G^{\star}\cong G $. \
\ $\Box$
\end{proof}

\begin{theorem}\label{th5}%------
For any $G\in \mathbb{K}_{s}^{t}$ and
$\tilde{G}\in\mathbb{K}^{n_{1},n_{2},\ldots, n_{k}}_{n}$, then they
are not $D$-cospectral.
\end{theorem}

\begin{proof}
We may assume that $G=K_{s}^{t}$ and
$\tilde{G}=K^{n_{1},n_{2},\ldots, n_{k}}_{n}$ with $m$ and
$\tilde{m}$ edges, respectively. Suppose that $G$ and $\tilde{G}$
have the same distance spectrum
$\lambda_{1}(D)\geq\lambda_{2}(D)\geq\cdots\geq\lambda_{n}(D)$. Let
$D(G)=(d_{ij})$ and $D(\tilde{G})=(\tilde{d_{ij}})$ be the distance
matrices of $G$ and $\tilde{G}$, respectively. Then we have
$$\sum_{l=1}^{n}\lambda_{l}^{2}=\sum_{i=1}^{n}\sum_{j=1}^{n}d_{ij}^{2}=2[m+(\frac{n(n-1)}{2}-m)\times 4+\frac{t(t-1)}{2}\times 5]=4n(n-1)-6m+5t(t-1)$$
and
$$\sum_{l=1}^{n}\lambda_{l}^{2}=\sum_{i=1}^{n}\sum_{j=1}^{n}\tilde{d_{ij}}^{2}=2[\tilde{m}+(\frac{n(n-1)}{2}-\tilde{m})\times 4]=4n(n-1)-6\tilde{m}.$$
Hence $6m-5t(t-1)=6\tilde{m}$, i.e., $6(m-\tilde{m})=5t(t-1)$. Note
that $t\geq 2$, then $m-\tilde{m}$ is a nonzero positive integer,
and hence $6|t(t-1)$.

If $s\geq t+1.$ By Lemmas \ref{le11} and \ref{le12}, we know that $-1$ is the
eigenvalue of $G$ and $\tilde{G}$ with multiplicity $s-t-1$ and
$n-k-1,$ respectively. Hence $s-t-1=n-k-1,$ then $k=2t,$ since
$s=n-t.$ By Corollary \ref{co22}, $\lambda_{2}(D(G))=\sqrt{2}-2\approx -0.5858,$
then $\lambda_{2}(D(\tilde{G}))=\sqrt{2}-2$. We claim that $k<6.$
Otherwise, the star $K_{1,6}$ is an induced subgraph of $\tilde{G}$,
by Lemma \ref{le2},
$\lambda_{2}(D(\tilde{G}))\geq\lambda_{2}(D(K_{1,6}))=-0.5678,$ a
contradiction. Note that $t\geq2$, then $k=2t=4$, contradicting
$6|t(t-1)$.

If $s=t.$ By Lemma \ref{le11}, $-1$ is the eigenvalue of $G$ with
multiplicity $0,$ hence $-1$ is also the eigenvalue of $\tilde{G}$
with multiplicity $0.$ By Lemma \ref{le12}, $k=n-1.$ Then
$n_{1}=n_{2}=\cdots=n_{k}=1,$ that is $\tilde{G}=K_{1,n-1}.$ Note
that $-2$ is the eigenvalue of $\tilde{G},$ then $-2$ is also the
eigenvalue of $G,$ we can deduce that $t^{2}-6t+2=0,$ a
contradiction. \ \ $\Box$
\end{proof}

\begin{theorem}\label{th6}%------
The connected graphs with
$\lambda_{2}(D(G))\leq\frac{17-\sqrt{329}}{2}\approx-0.5692$ are
determined by their $D$-spectra.
\end{theorem}
\begin{proof}
By Lemma \ref{le5}, $G\in K_{n}\cup \mathbb{K}_{s}^{t}\cup
\mathbb{K}^{n_{1},n_{2},\ldots, n_{k}}_{n}.$ Obviously, $K_{n}$ is
determined by its $D$-spectrum. By Theorems \ref{th3}, \ref{th10} and
\ref{th5}, the result follows immediately.
\end{proof}

%\noindent{\bf Acknowledgment}

%The authors would like to thank the anonymous referees very much for
%valuable suggestions and corrections which improve the original
%manuscript.

\small {

}
\end{document}